\documentclass[12pt]{article}
\usepackage[margin=1in]{geometry}
\usepackage{geometry,amssymb,amsmath,enumerate, tikz, hyperref}
\usepackage{amsfonts}
\usepackage{amssymb}
\usepackage{amsmath}

\newfont{\msbm}{msbm10 scaled\magstephalf}

\newtheorem{theorem}{Theorem}[section]
\newtheorem{corollary}[theorem]{Corollary}

\newtheorem{lemma}[theorem]{Lemma}

\newtheorem{definition}[theorem]{Definition}

\newtheorem{claim}[theorem]{Claim}

\newtheorem{conjecture}[theorem]{Conjecture}

\def\b1K{\mbox{\boldmath $K$}_{-1}}
\def\bK{\mbox{\boldmath $K$}}

\newbox\noforkbox \newdimen\forklinewidth
\forklinewidth=0.3pt \setbox0\hbox{$\textstyle\smile$}
\setbox1\hbox to \wd0{\hfil\vrule width \forklinewidth depth-2pt
 height 10pt \hfil}
\wd1=0 cm \setbox\noforkbox\hbox{\lower 2pt\box1\lower
2pt\box0\relax}

\newcommand{\EE}{\mbox{\msbm E}}

\def\sub'm{\prec_{\bK'}}
\def\grpf #1 #2{{\rm grp}_{#2}(#1)}
\def\fldf #1 #2{{\rm fld}_{#2}(#1)}
\def\dclf #1 #2{{\rm dcl}_{#2}(#1)}
\def\rclf #1 #2{{\rm rcl}_{#2}(#1)}
\def\aclf #1 #2{{\rm acl}_{#2}(#1)}
\def\acff #1 #2{{\rm acf}_{#2}(#1)}
\def\strf #1 #2{{\rm str}_{#2}(#1)}
\def\tclf #1 #2{{\rm acf}_{#2}(#1)}

\def\hbar{{\bf h}}

\date{\today}

\newcommand{\F}{\mathcal{F}}

\newcommand{\Y}{\mathcal{Y}}

\newcommand{\e}{\mathrm{e}}
\newcommand{\vr}{\mathrm{v}}

\newcommand{\Hc}{\mathcal{H}}

\newcommand{\ve}{\varepsilon}

\newcommand{\ex}{\mathrm{ex}}

\usepackage{mathtools}

\newtheorem{prob}{Problem}

\newenvironment{proof}{\paragraph{Proof:}}{\hfill$\square$\\}

\begin{document}

\title{Supersaturation via edge-gluing}
\author{Zihao Jin\footnote{Dept. of Mathematics, University of San Diego,  San Diego, CA 92093, USA {\tt zij013@ucsd.edu}}, Sean Longbrake \footnote{Dept. of Mathematics, Emory University,  Atlanta, GA 30322, USA {\tt sean.longbrake@emory.edu}}, Liana Yepremyan \footnote{Dept. of Mathematics, Emory University,  Atlanta, GA 30322, USA {\tt lyeprem@emory.edu}}
\thanks{The research of the first author was partially supported by the independent research program for undergraduate students at Emory University. The research of the second author partially and the third author fully are supported by the National Science Foundation grant DMS-2247013: Forbidden and Colored Subgraphs.} }

\maketitle

\abstract{
In 1984, Erd\H{o}s and Simonovits conjectured the following: given a bipartite graph $H$, there exist constants $\beta, C > 0$ such that any graph $G$ on $n$ vertices and $pn^2\geq C \ex(n, H)$  edges contains at least $\beta n^{\vr(H)} p^{\e(H)}$ copies of $H$. We show that edge-gluing preserves the satisfiability of this conjecture under some mild symmetry conditions. Namely, if two graphs $H_1$ and $H_2$ satisfy this conjecture, and if furthermore, gluing them along a fixed edge produces a unique graph then the resulting graph satisfies the conjecture as well. 

In the same paper, Erd\H{o}s and Simonovits conjectured a weaker statetement: for every $H$, there is some $\alpha, \beta, C > 0$ such that any graph $G$ on $n$ vertices and $pn^2\geq C n^{1+ \alpha}$  edges contains at least $\beta n^{\vr(H)} p^{\e(H)}$ copies of $H$. We show that if $H$ satisfies this conjecture then by gluing several copies of labeled $H$ along the same copy of a subforest of $H$ produces a graph that also satisfies the conjecture.
 }

\maketitle
\section{Introduction}

 The \emph{Tur\'an number} of a graph $H$, denoted by $\ex(n, H)$, is the maximum possible number of edges in an $n$-vertex graph that does not contain a copy of $H$.  The famous
Erd\H{o}s-Simonovits-Stone theorem states that any graph $H$ satisfies $$\ex(n,H) = \left(1-\frac{1}{\chi(H)-1}\right){n \choose 2} +o(n^2),$$ where $\chi(H)$ is the chromatic number of $H$. Thus the Tur\'an number of any graph $H$ is asymptotically determined, except when $\chi(H)=2$, that is, when the graph is bipartite.  For bipartite $H$, even the order of magnitude of $\ex(n,H)$  is known only for very few graphs/families of graphs; for example, the only cycles we know the right order of magnitude for are $C_4$, $C_6$ and $C_{10}$.

Given our limited knowledge of the bipartite Tur\'an numbers, it is natural to study how certain graph operations affect these numbers. To this end, Erd\H{o}s and Simonovits~\cite{erdos1970some} showed that given a bipartite graph $H$ with $\ex(n, H) = O(n^\alpha)$ and bipartition $(A,B)$, the graph $H'$ formed by joining $A$ to one side of a complete bipartite graph, $K_{t, t}$ and joining $B$ to the other, satisfies $\ex(n, H') = O(n^{\beta})$ where $\beta$ is the unique solution to $\frac{1}{\alpha} = \frac{1}{\beta} + t$. Faudree and Simonovits~\cite{faudree1983class} showed that if a graph $H$ satisfies $\ex(n, H) = O(n^{2 - \alpha})$, then the graph formed by identifying the leaves of the $(k-1)$-subdivision of the star $K_{1, s}$ to one of the sides of $H$ satisfies that $\ex(n, H') = O(n^{2 - \beta})$ for $\beta = \frac{\alpha + \alpha^2 + \dots + \alpha^{ k-  2}}{1 + \alpha + \alpha^2 + \dots + \alpha^{ k-  2}}$, where a $k$-subdivision of a graph is obtained by replacing all edges with internally vertex disjoint paths of length $k+1$. In a similar vein, Kupavskii, Polyanskii, Tomon, and Zakharov \cite{kupavskii2022extremal} considered the operation of \emph{gluing} two edges for any two $r$-uniform hypergraphs, when we identify two edges as the same. Note that the gluing operation produces a family of graphs, and not a single graph. In particular for $r=2$, this family has size two. The following result is implicit in their proof of Lemma 4.1, where instead of considering gluing along two fixed edges, they glue along every possible pair of edges. 

\begin{theorem}[Lemma 4.1 in~\cite{kupavskii2022extremal}]
Let $H_1$ and $H_2$ be two bipartite graphs with $\ex(n, H_1)~=~O(n^{1 + \alpha})$ and $\ex(n, H_2)~=~O(n^{1 + \beta})$. Fix $e_1 \in E(H_1)$ and $e_2 \in E(H_2)$. Then, the family $\Hc$ of graphs formed by identifying $e_1$ and $e_2$ satisfies
$$\ex(n, \Hc) = O(n^{1 + \max\{\alpha, \beta\}}).$$
\end{theorem}

Similarly, Dong, Gao, and Liu~\cite{dong2025bipartite} considered the operation of gluing two graphs on a fixed vertex and proved the following. 

\begin{theorem}[Theorem 1.5 in \cite{dong2025bipartite}]
Let $H$ be a bipartite graph with bipartition $(A, B)$. Given two vertices $u, v \in A$ or $u,v \in B$, let $H_{u,v}$ be formed from two copies $H_1, H_2$ by identifying  $u$ with $v$.
Then, 
$$\ex(n, H_{u,v}) = \Theta(\ex(n, H)).$$
\end{theorem}

In this paper, we are interested in the supersaturation questions for bipartite graphs, that is, if we are above the Tur\'an number, how many copies of a graph $H$ can we find in a host graph $G$? In this direction, the famous conjecture of Sidorenko~\cite{sidorenko1991inequalities} states the following: 
\begin{conjecture}[\cite{sidorenko1991inequalities}]
Let $H$ be a bipartite graph and $G$ be any graph with $pn^2$ edges. Then the number of homomorphisms from $H$ to $G$ is at least $p^{\e(H)}n^{\vr(H)}$. 
\end{conjecture}

It is easy to see that the number of homomorphic copies in this conjecture is tight for the random binomial graph  $G(n,p)$ for most values of  $p$, more specifically, when the number of isomorphic copies of $H$ is of the same order as the number of homomorphic copies. Sidorenko's Conjecture has been studied greatly with it being verified for many different bipartite graphs, such as trees \cite{sidorenko1991inequalities}, bipartite graphs with one vertex in one part complete to the other part \cite{conlon2010approximate}, among other results~\cite{conlon2018some, conlon2018sidorenko, conlon2018sidorenko2, kim2016two, sidorenko1991inequalities, sidorenko1993correlation}. Among them, we highlight the following theorem of Li and Szegedy~\cite{li2011logarithimic}: 
\begin{theorem}[Theorem 2 in \cite{li2011logarithimic}]\label{liszegedy}
 Let $H_1$, $H_2$ be two graphs satisfying Sidorenko's Conjecture and $e_1 \in E(H)$ and $e_2 \in E(H)$
arbitrary edges. Then both graphs obtained from $H_1$ and $H_2$ by identifying $e_1$ and $e_2$ also satisfy Sidorenko's Conjecture.
\end{theorem}

This result says that the class of graphs for which we know Sidorenko's Conjecture for is closed under edge-gluing operation. Our goal in this paper is to extend this to the following closely related conjectures of Erd\H{o}s and Simonovits \cite{ES-cube}.

\begin{conjecture}[Conjecture $2$ in \cite{ES-cube}]\label{es-conj1}
For every bipartite graph $H$ there exists $\beta, C > 0$ such that the following holds. If $e(G) = pn^2 \geq C \ex(n, H)$, then the number of embeddings of $H$ into $G$ is $\beta p^{e(H)}n^{v(H)}$. 
\end{conjecture}

\begin{conjecture}[Conjecture $2^*$ in ~\cite{ES-cube}]\label{es-conj2}
For every bipartite graph $H$ there exists $\beta, C, \alpha > 0$ such that the following holds. If $e(G) = pn^2 \geq C n^{1 + \alpha}$, then the number of embeddings of $H$ into $G$ is $\beta p^{e(H)}n^{v(H)}$. 
\end{conjecture}

It is not hard to see that if a graph satisfies Sidorenko's Conjecture, then it  satisfies Conjecture~\ref{es-conj2} with $\alpha = 1 - \frac{1}{e(H)}$ because for $p\geq Cn^{- \frac{1}{e(H)}}$ most of the homomorphic copies of $H$ in $G$ are isomorphic copies. Conjecture~\ref{es-conj1} is stronger than Sidorenko's conjecture above the Tur\'an threshold of a graph $H$. Indeed, if a bipartite graph $H$ has Tur\'an number  of order exactly $\Theta(n^{1 + \alpha})$~\footnote{This is believed to be true for every bipartite graph, see for example \cite{erdHos1982compactness}.}, then Conjecture~\ref{es-conj1} states for $p\geq Cn^{\alpha - 1}$ the number of isomorphic copies of $H$ is of order $p^{\e(H)}n^{\vr(H)}$. 

Conjecture~\ref{es-conj1} is very hard, and only known for a few families of graphs due to the sparsity of known tight Tur\'an numbers. Thus, we aim to establish Conjecture~\ref{es-conj2} for a given $H$, for the best known $\alpha$ for which $ \ex(n, H) =O( n^{1 + \alpha})$. For such $\alpha$, Conjecture~\ref{es-conj2} has been established for the cube \cite{ES-cube}, cycles \cite{jiang2020supersaturation,FS}, theta graphs~\cite{corsten2021balanced}, complete bipartite graphs~\cite{ES-supersat}, and the families discussed in these works \cite{conlon2010approximate, grzesik2022turan, jiang2023tree}.  In the current paper, we study how edge or subgraph gluing affects  the satisfiability of  Conjecture~\ref{es-conj1}.
If $H$ satisfies Conjecture~\ref{es-conj2} for specific value of $\alpha$, we will call such a graph $\alpha$-supersaturated.  Our main result follows.

\begin{theorem}\label{maingluingalongedge} For $0\leq \alpha<1$, let $H_1$ and $H_2$ be two $\alpha$-supersaturated graphs. Fix $e_1 \in E(H_1)$ and $e_2 \in E(H_2)$. Let $\Hc$ be the family of graphs formed by identifying $e_1$ and $e_2$. Then there exist constants $C, \beta> 0$, such that for any graph $G$ with $pn^2\geq Cn^{1+\alpha}$ edges, the number of copies of some member of $\Hc$ in $G$ is at least $\beta p^{\e(H_1) + \e(H_2) - 1}n^{\vr(H_1) + \vr(H_2) - 2}$ for some $\beta$ depending only on $H_1, H_2$. 
\end{theorem}

This implies that when $\ex(n, H_1) = \Theta(n^{1+ \alpha})$ and $\ex(n, H_2) = O(n^{1+ \alpha})$, and $\Hc = \{H\}$, $H$ must satisfy  Conjecture~\ref{es-conj1}. In particular, we are able to give the best known supersaturation bounds for cycles with a chord. Recall that  $C_{2k}$ is $\frac{1}{k}$-supersaturated~\cite{jiang2020supersaturation, FS}.  Let $H$ be formed by gluing a $C_{2k}$ and a $C_{2 \ell}$ with $\ell \geq k$, then this implies $H$ is $\frac{1}{k}$-supersaturated. This is tight, for example, for $k=2,3,5$ due to having tight bounds for corresponding extremal numbers of $C_4, C_6$ and $C_{10}$~\cite{brown1966graphs, lazebnik1994properties, lazebnik1999polarities}.

We can also obtain optimal results for the family of \emph{cycle blow-ups of trees}. Let $T$ be a tree on $[n]$ and $\mathcal{C} = \{C_1, \dots C_n\}$ be a family of vertex-disjoint even cycles. For every edge $ij \in E(T)$, we mark a pair of vertices $v_{ij}^i \in C_i$ and $v_{ij}^j \in C_j$. Let $T[\mathcal{C}]$ be the graph formed from the graphs $C_i \in \mathcal{C}$ by connecting the vertices $v_{ij}^i$ and $v_{ij}^j$ for every $ij \in E[T]$ by vertex-disjoint paths $P_{ij}$ of arbitrary length. In case the length of the path $P_{ij}$ is zero, we simply glue the two vertices  $v_{ij}^i$ and $v_{ij}^j$. Theorem~\ref{maingluingalongedge} implies the following:

\begin{theorem}\label{cycle-blowups}
Let $\ell$ be any positive integer. Given a tree $T$ on $[n]$ and $\mathcal{C} = \{C_1, \dots C_n\}$ be a family of vertex-disjoint even cycles with $|C_i| \geq 2 \ell$ for all $i\in [n]$, the graph $T[\mathcal{C}]$ is $\frac{1}{\ell}$-supersaturated. 
\end{theorem}
This result is in the spirit of the results  of Grzesik, Janzer, and Nagy for $(r, t)$-blowups of trees \cite{grzesik2022turan} and Jiang and the second author for tree-degenerate graphs \cite{jiang2023tree}.  Note that if we believe the conjecture of Bondy and Simonovits \cite{BondySim74} that $\ex(n,C_{2k})=\Theta(n^{1+1/k})$, then this family $T[\mathcal{C}]$ satisfies Conjecture~\ref{es-conj1}. 

In a slightly different direction, Szegedy~\cite{szegedy2014information} roughly speaking showed that if a graph $H$ can be formed by gluing graphs that satisfy Sidorenko's Conjecture along forests, then $H$ satisfies Sidorenko's Conjecture. In a  similar fashion, we show that if $H$ can be formed by an $\alpha$-supersaturated graph $F$ by gluing many copies of $F$ along a fixed tree, it will be $\alpha'$-supersaturated for some $\alpha'$.

\begin{theorem}\label{mainmain}
Let $H$ be an $\alpha$-saturated graph with $h$ vertices, and let $F$ be a proper forest contained in $H$ with $\ell$ vertices. Let $H^*$ be formed from $H$ by taking $s$ copies of $H$ and gluing them along $F$. Let $\alpha' = \max\{ 1 - \frac{h - \ell}{\e(H) - \e(F)}, 1 - \frac{1 - \alpha}{\ell - 1 + (1 - \e(F))(1 - \alpha)}\}$. 
Then, there exists a constant $C$ such that
$$ \ex(n, H^*) \leq Cn^{1 + \alpha'} $$ and moreover, $ H^*$ is $\alpha'$-supersaturated.
\end{theorem}

 Given a graph $H$ and a set $S$ of vertices, we let $e(S)$ denote the number of edges of $H$ with at least one endpoint in $S$ and let $\rho(S)=\frac{e(S)}{|S|}$. Given a tree $T$, let $R$ be the set of all leaves and define $\rho(T)=\rho(V(T)\setminus R)$.
 We call a tree $T$ {\it $\rho$-balanced} if for every $S\subseteq V(T)\setminus R$, we have $\rho(S)\geq \rho(T)$. For instance, any path of length at least two is $\rho$-balanced as is any star with at least two leaves. Fix a $\rho$-balanced tree $T$. For any positive integer $q$, let ${\mathcal T}^q_R$ denote the family of graphs obtainable by taking the union of $q$ distinct copies of $T$ each of which agree on $R$. A simple deletion method using Erd\H{o}s-Renyi random graphs shows that for any $\varepsilon>0$, there exists $q_0$ such that for all $q\geq q_0$, $\ex(n,{\mathcal T}^q_R)= \Omega(n^{2-1/\rho(T)-\varepsilon})$. Bukh and Conlon \cite{BC} used a probability space based on random polynomials over finite fields to show that in fact, there exists $q_0$ such that for all $q\geq q_0$, $\ex(n,{\mathcal T}^q_R)= \Omega(n^{2-1/\rho(T)})$.  On the other hand, a simple counting argument shows that $\ex(n,{\mathcal T}^q_R)=O(n^{2-1/\rho(T)})$.  If we restrict our attention to the graph $T_R^q$, the single graph which is formed by taking a union of $q$ distinct copies of $T$ which agree only on $R$, the lower bound still holds, but the upper bound becomes a difficult question to which much work has been given \cite{Conlon_Janzer_2022, Conlon_Janzer_Lee_2020, Conlon_Lee_2021, Janzer_2018, Janzer_2020, Jiang_Jiang_Ma_2022, Jiang_Ma_Yepremyan_2022, Jiang_Qiu_2023, Kang_Kim_Liu_2018}. It was conjectured by Bukh and Conlon~\cite{BC} that $\ex(n,T^q_R)=O(n^{2-1/\rho(T)})$ must hold. We observe that this graph $T_R^q$ satisfies Sidorenko's Conjecture for all $q$, and thus is $(1 - \frac{1}{e(T_{R}^q)})$-supersaturated. For theta graphs~\cite{corsten2021balanced}, complete bipartite graphs~\cite{ES-supersat} and their $2$-subdivisions~\cite{Qiu}, a better $\alpha$ is known. As a corollary to Theorem~\ref{mainmain}, we show that for all trees and all $q$, $\alpha$ can be taken to be $(1 - \frac{1}{|R|})$. 

 \begin{corollary}\label{maintree}
Let $T$ be a tree on $t \geq 3$ vertices with $r$ leaves,  and let $T^q$ be formed by gluing $q$ copies of $T$ along all the leaves.  Then $T^q$ is $(1 - \frac{1}{r})$-supersaturated. 
\end{corollary}
This result is tight for large $q$ and $T = K_{1, r}$, but is likely far from the truth in most cases. We believe that one should be able to take $\alpha=1 - \frac{1}{\rho(T)}$ when $T$ is $\rho$-balanced. If true, this would be a strengthening of the Bukh-Conlon conjecture mentioned above.

\section{Notation and Preliminary results}

For a graph $G$, let $\delta(G)$ and $\Delta(G)$ denote its minimum and maximum degrees.  For $K\geq 0$ real number, we say a graph $G$ is $K$-almost-regular if $\Delta(G) \leq K\delta(G)$.

     \begin{definition} For some real numbers $ \alpha, \beta, C > 0$, a graph $H$ is $(\alpha, C, \beta)$-supersaturated if for every graph with $pn^2$ edges, satisfying $p \geq Cn^{\alpha - 1}$, $G$ has at least $\beta p^{\e(H)}n^{\vr(H)}$ many copies of $H$.
     \end{definition}

This section states some auxiliary lemmas, including a standard reduction to $K$-almost-regular graphs, and a novel \textquotedblleft cleaning" lemma which is the key tool of this paper. Roughly speaking, it is in the spirit of ``balanced supersaturation'' and says we can find many copies of a given graph $H$ in a way that no fixed edge is in many copies, and furthermore, if we condition copies containing a specific subgraph $F$ of $H$ then no vertex $v$ outside of $F$ is involved in many such copies.

\begin{theorem}[Jensen's Inequality~\cite{jensen1906}]
\label{Jensen's Inequality}
Let $\varphi$ be a real convex function. Let $x_1, \dots , x_n$ be numbers from its domain, and $a_1, \dots ,a_n$ positive real numbers. Then $$\varphi \left( \frac{\sum_{1\leq i \leq n} a_ix_i}{\sum_{1\leq i \leq n}a_i} \right)\leq \frac{\sum_{1\leq i \leq n}a_i\varphi(x_i)}{\sum_{1\leq i \leq n} a_i}.$$

If $\varphi$ is a concave function, then the inequality is reversed.
\end{theorem}

The following lemma is Theorem 3.3 from Jiang and the third author~\cite{jiang2020supersaturation}, which allows us to reduce a saturation problem in the general setting to one with $K$-almost-regular host graphs.

\begin{theorem}[Theorem 3.3 in~\cite{jiang2020supersaturation}] \label{reduction} 
Let $H$ be a graph containing a cycle. Then, there exist constants $K,C'>0$ such that if every $K$-almost-regular graph $G$ on $n$ vertices with $\e(G) \geq C  n^{ 1 + \alpha}$ edges contains at least $\beta \e(G)^{\e(H)}n^{\vr(H) - 2 \e(H)}$ copies of $H$, then $H$ is $(\alpha, \max\{4C, C'\}, \frac{\beta}{4^{\e(H)}})$-supersaturated. 
\end{theorem}

We will also use this standard approximation lemma.  
\begin{lemma}\label{bounds}
For any $s$ and $q \leq 1$ such that $qn \geq 2 ( s - 1)$
$$q \left(\frac{q}{2}\right)^{s - 1} \leq \frac{\binom{n - s}{qn - s}}{\binom{n}{qn}} \leq q^s.$$  
\end{lemma}

The following is a common probabilistic lower bound that would be useful in both the proof and applications of Lemma~\ref{first-cleaning}. A result of Bohman and Keevash shows the bound below is tight if and only if $H$ and $F$ are trees \cite{bohman2010early}\footnote{Notice that if the result is tight then we may assume $F$ is strictly $2$-balanced.}.

\begin{theorem}\label{problowerbound}
For all bipartite graph $H$ and $F$ a subgraph of $H$ with at least two edges, 
$$\ex(n, H) = \Omega (n^{2 - \frac{\vr(F) - 2}{\e(F) - 1}}).$$
In particular, for all subgraphs $F \subset H$ with at least one edge, if there exists an $C$ such that $\ex(n, H) \leq Cn ^{ 1 + \alpha}$,  then  $$(\vr(F) - 1) + (1 - \e(F))(1 - \alpha) \geq 1.$$
\end{theorem}

\begin{proof}
Let $p = \frac{1}{4} n^{ - \frac{\vr(F) - 2}{\e(F) - 1}}$ and $G = G(n,p)$ be a random graph such that each edge is kept with probability $p$. Let $X$ be the random variable counting the number of labeled copies of $F$ in $G$. Then $\EE[X] \leq n^{\vr(F)}p^{\e(F)}$. Meanwhile, linearity of expectation gives $\EE[\e(G)] = p\binom{n}{2}$. Let $G'$ be the subgraph of $G$ obtained by removing one edge from every labeled copy of $F$. Then \begin{align*}
\EE[\e(G')] &\geq p\binom{n}{2} - n^{\vr(F)}p^{\e(F)}\\
&\geq p \binom{n}{2} - n^{\vr(F)}p n^{\vr(F) - 2}\\
&\geq p \binom{n}{2} -\frac{1}{4} n^{2} p\\
&\geq \frac{1}{2}p\binom{n}{2}\\
&= \Omega( n^{2 - \frac{\vr(F) - 2}{\e(F) - 1}}). 
\end{align*}

Hence there exists an $F$-free (and thus $H$) graph with $\Omega(n^{2-\frac{\vr(F)-2}{\e(F)-1}})$ edges, and the lower bound follows from the fact that $F$-free graphs are $H$-free. Furthermore, if $\ex(n,H) \leq Cn^{1+\alpha}$, then we must have $$2-\frac{\vr(F)-2}{\e(F)-1} \leq 1+\alpha$$ and the desired inequality follows.
\end{proof}

Finally, we introduce our main cleaning result, an adaptation of similar results from \cite{FMS} and \cite{JL2}, however in difference to these results we control the number of \emph{vertices} involved, instead of the \emph{edges}. Note that $\ve$ may not be less than one, although we will apply it with $\ve = 1$ later on.

\begin{lemma}\label{first-cleaning}
Let $H$ be a graph on $h$ vertices which is $(\alpha, C, \beta)$-supersaturated, and $F$ be a subgraph of $H$ on $\ell$ vertices such that $0<\ell< h$ and $0<\e(F)<\e(H)$. Let $f \in E(H)$ be a fixed edge.

Then there exist constants $C', \beta', n_0 >0$ depending only on $H$ and $F$, such that for any graph $G$ on $n \geq n_0$ vertices and for any positive real numbers $\ve, \gamma, p$  satisfying $\gamma<1$ and
\begin{align*} 
p  &\geq g(n, H, F, \ve, \gamma) := C' \max(n^{-\frac{h-\ell}{\e(H)-\e(F)}}\gamma^{-\frac{1}{\e(H) - \e(F)}}, (n\gamma)^{\frac{\alpha -1}{(\ell-1)+(1-\alpha)(1-\e(F))}})
\end{align*} the following holds. If $\e(G) = (1 + \frac{2}{\ve}) p n^2$, then there is a family $\mathcal{H}$ of embeddings of $H$ in $G$ such that \begin{enumerate}
\item $|\Hc| \geq \beta' p^{\e(H)} n^{h}.$
\item For all embeddings $\psi: F \rightarrow G$ and all vertices $u$ not in $\psi(F)$, there are at most $ \gamma p^{\e(H) -\e(F)}n^{h - \ell}$ embeddings $\varphi \in \Hc$ such that $\varphi|_F = \psi$ and $u \in \varphi(H)$.
\item For any edge $e  \in E(G)$, there are at most $(1 + \ve) \frac{\beta' p^{e(H) } n^h}{e(G)}$ members $\varphi$ in $\Hc$ such that $\varphi(f) = e$. 
\end{enumerate}
\end{lemma}

\begin{proof}

  Let $\Hc = \emptyset$ and $\beta' = \frac{\beta}{2^{3\e(H)}}$. Clearly, $\Hc$ satisfies item $(2), (3)$. Throughout the proof, we would interpret and refer to embeddings of $H$ into $G$ as copies of $H$ in $G$. We will slowly add copies of $H$ to $\Hc$ maintaining (2), (3) until $\Hc$ has the desired size. The main idea is to look at a collection of ``patches" of some fixed size which avoids the copies of $F$ that have been covered many times already, and then recruit a new copy of $H$ from this collection.

So assume $|\mathcal{H}| < \beta' p^{\e(H)} n^h$ and satisfies $(2), (3)$. We want to find an additional copy of $H$ which we can recruit so that the family $\mathcal{F}$ would continue to retain properties $(2), (3)$.

To begin, let $B$ be the collection of edges $e  \in E(G)$ such that the number of $\varphi$ in $\mathcal{H}$ with $\varphi(e) = f$ is at least $ \frac{\varepsilon \beta '}{2} p^{\e(H) - 1} n^{h - 2} > (1+\frac{\varepsilon}{2})\frac{|\mathcal{H}|}{\e(G)}$. Then, we obtain the following upper bound for the size of $B$:

\begin{claim}
$|B| < \frac{2}{\ve} pn^{2}$.
\end{claim}

\begin{proof}
Let us consider $\mu$, the number of pairs $(e, \varphi)$ with $\varphi \in \Hc$ and $e \in B$ with $\varphi(f) = e$.  On one hand, we have a lower bound $\mu \geq \frac{\varepsilon \beta '}{2} p^{\e(H) - 1} n^{h - 2} |B|$. On the other hand,  we have an upper bound of the form $$\mu \leq |\Hc| \leq \beta '  p^{\e(H)} n^h.$$ Combining the two bounds gives the desired upper bound for $|B|$.
\end{proof}

Let $G'$ be the graph formed by removing every edge in $B$. Observe that $G'$ satisfies $pn^2 \leq e(G') \leq (1 + \frac{2}{\ve} )pn^2$.

For a given $\psi: F \rightarrow G'$ and $u \not \in \psi(F)$, let $\deg_{\mathcal{H}}(\psi, u)$ denote the number of $\varphi$ in $\mathcal{H}$ with $\varphi|_{H} = \psi$ and $u \in \varphi(H)$. Let $S$ be the collection of pairs $(\psi, u)$ with $\deg_{\mathcal{H}}(\psi, u) \geq \frac{\gamma}{2} p^{\e(H) - \e(F)}n^{h - \ell}$. Then we immediately obtain the following upper bound for the size of $S$.
\begin{claim}
$|S| \leq \frac{2 \beta ' h}{\gamma} p^{\e(F)}n^{\ell}.$
\end{claim}
\begin{proof}
Let $\mu$ be the number of pairs $((\psi, u), \varphi)$ with $\varphi \in \F$ and $(\psi, u) \in S$ with $\varphi|_{F} = \psi$ and $u \in \varphi(H)$. On one hand, we have 
$$\frac{\gamma}{2} p^{\e(H) - \e(F)} n^{h-\ell} |S| \leq \mu. $$ On the other hand, we have that $$\mu \leq h |\Hc| < \beta ' h p^{\e(H)}n^h$$ for there are at most $h$ many choices for $u$ inside $\varphi(H)$ for each $\varphi \in \mathcal{H}$. Combining the two bounds gives the desired upper bound for $|S|$.
\end{proof}

The plan for the rest of the proof is to randomly sample patches of a certain size, $qn$, from which we could further obtain subgraphs of high density and containing no members of $S$, and finally count potential new copies of $H$ by applying $(\alpha, C, \beta)$-supersaturatedness to each patch and recruit a new copy of $H$.

So let $0\leq q\leq 1$ be a rational number to be determined later.  For every set $W\subseteq V(G)$ of size $qn$, let $X(W)$ be the the number of $(\psi, u)\in S$ such that $\psi, u$ are in $W$.  Let $Y$ be the graph formed from $G'[W]$ by removing one edge from $\psi(F)$ for each $S$-member $(\psi, u)$ for which $\psi(F)$ and $u$ are both in $W$.   

Let $W_0$ be a random subset of $V(G)$ of size $qn$ chosen with uniform probability from the sample space of all subsets of size $qn$, and $X_0, Y_0$ be defined as $X_0=X(W_0)$ and $Y_0=Y(W_0)$ as random variables.  Thus, by Lemma~\ref{bounds}, we have that $\EE[\e(G'[W_0])] = pn^2\frac{\binom{n-2}{qn-2}}{\binom{n}{qn}} \geq \frac{q^2pn^2}{2}$. Then, similarly by Lemma~\ref{bounds}, $\EE[X_0] = \frac{\binom{n - \ell - 1}{qn - \ell - 1}}{\binom{n}{qn}}|S|\leq q^{\ell + 1} |S|$ as each element in $S$ spans $\ell + 1$ vertices. And thus, $\EE[\e(Y_0)] \geq \frac{q^2pn^2}{2} - \frac{2 \beta' h}{\gamma} q^{\ell + 1} p^{\e(F)}n^{\ell}$ for all sufficiently large $n$. We seek to find a value of $q$ so that this inequality holds: 
\begin{equation*}
 \frac{q^2pn^2}{2} - \frac{2 \beta' h}{\gamma} q^{\ell + 1} p^{\e(F)}n^{\ell} \geq 2C(qn)^{1 + \alpha}
\end{equation*} 
Yet, as long as 
\begin{align*}
\frac{2 \beta' h}{\gamma} q^{\ell + 1} p^{e(F)}n^{\ell} &\leq \frac{1}{4} q^2 pn^2 \\
q^{\ell - 1} &\leq \frac{\gamma}{8 \beta' h} p^{1 - e(F)} n^{2 - \ell} \\
q &\leq \left( \frac{\gamma}{8 \beta' h}  \right)^{\frac{1}{\ell - 1}} p^{\frac{1 - \e(F)}{\ell - 1}}n^{-1 + \frac{1}{\ell - 1}}
\end{align*}
we would have the following bound
\[ \EE[\e(Y)] \geq \frac{1}{4}q^2 pn^2. \] 

This is at least $2C(qn)^{1 + \alpha}$ when: 

\begin{align*}
\frac{1}{4}q^2 pn^2 & \geq 2C(qn)^{1 + \alpha}\\
q &\geq (8C)^{\frac{1}{1 - \alpha}} p^{-\frac{1}{1 - \alpha}} n^{-1}
\end{align*}

Observe that $(16C)^{\frac{1}{1 - \alpha}} p^{-\frac{1}{1 - \alpha}} n^{-1} \geq 2 \cdot (8C)^{\frac{1}{1 - \alpha}} p^{-\frac{1}{1 - \alpha}} n^{-1}$ since $0<1-\alpha<1$. Thus, a choice of $q$ which makes $qn$ an integer exists as long as $q\geq (16C)^{\frac{1}{1 - \alpha}} p^{-\frac{1}{1 - \alpha}} n^{-1}$ in conjunction with the lower bound we have obtained earlier.
\begin{align*}
(16C)^{\frac{1}{1 - \alpha}} p^{-\frac{1}{1 - \alpha}} n^{-1} &< \left( \frac{\gamma}{ 8 \beta' h}  \right)^{\frac{1}{\ell - 1}} p^{\frac{1 - \e(F)}{\ell - 1}}n^{-1 + \frac{1}{\ell - 1}} \\
p^{-\frac{1}{1 - \alpha} - \frac{1 - \e(F)}{\ell - 1}} &<   \left( \frac{ \gamma}{ 8\beta' h}  \right)^{\frac{1}{\ell -1}}  (16 C)^{- \frac{1}{1 - \alpha}}n^{\frac{1}{\ell- 1}}\\
p &> \left( (16 C)^{\ell - 1}\left( 8 \beta' h  \right)^{ 1 - \alpha}\right)^{ \frac{1}{(\ell - 1) + (1 - \e(F))(1 - \alpha)}}  (n\gamma) ^{ - \frac{1 - \alpha}{(\ell - 1) + (1 - \e(F))(1 - \alpha)}}
\end{align*}
since then for all sufficiently large $n$, the interval
\begin{align*}
((8 C)^{\frac{1}{1-\alpha}}p^{-\frac{1}{1-\alpha}},\left( \frac{\gamma}{ 8 \beta' h}  \right)^{\frac{1}{\ell- 1}} p^{\frac{1 - \e(F)}{\ell - 1}}n^{\frac{1}{\ell- 1}}) 
\end{align*}
contains an integer greater than $2$, as the upper bound of the interval is certainly at least twice the lower bound, and the lower bound is at least $2$. Note that the inequality flips because $(\ell - 1) + (1 - \e(F)) ( 1 - \alpha) \geq 1$ by Lemma~\ref{problowerbound}. 

$$p > \left( (16 C)^{\ell - 1}\left( 8 \beta' h \right)^{ 1 - \alpha}\right)^{ \frac{1}{(\ell - 1) + (1 - \e(F))(1 - \alpha)}}  (n\gamma) ^{ - \frac{1 - \alpha}{(\ell - 1) + (1 - \e(F))(1 - \alpha)}}.$$
Notice that we have $$\EE[\e(Y_0)] \cdot\binom{n}{qn}=\sum_{W\subseteq V(G):|W|=qn}{\e(Y)}.$$ Let $\mathcal{Y}_1$ be the set of subgraphs  $Y(W)$ with at least $C(qn)^{1 + \alpha}$ edges and $\mathcal{Y}_2$ be the set of such subgraphs with strictly less than so many edges.

We have that the sum of edges over all $Y$ by the earlier expectation calculation is at least $2 \binom{n}{qn} C(qn)^{1 + \alpha}$.  On the other hand,  the sum over the edges in $\Y_2$ must be strictly less than $ \binom{n}{qn} C(qn)^{1 + \alpha}$,  so at least half the edges in the total count are captured by the members in $\Y_1$. Thus, in particular 
$$\sum_{Y \in \Y_1} \e(Y) \geq \frac{1}{8} \binom{n}{qn} np^2q^2. $$

On the other hand, for each $Y \in \Y_1$,  the number of $\varphi: H \rightarrow G'$ contained in $Y$ is at least $$\beta (qn)^{h}\left(\frac{\e(Y)}{q^2n^2}\right)^{\e(H)}$$ by $H$ being $(\alpha, C, \beta)$-supersaturated. Thus, noting that any $\varphi$ gets counted at most $\binom{n - h}{qn - h}$ times,  we have the following count $\mu$, the number of $\varphi: H \rightarrow G'$ that miss any bad configuration: 
\begin{align*}
\mu \binom{n - h}{qn - h} &\geq \beta \sum_{Y\in \Y_1} (qn)^{h}\left(\frac{\e(Y)}{q^2n^2}\right)^{\e(H)}\\
\mu \binom{n - h}{qn - h} &\geq \beta (qn)^{h - 2 \e(H)} \binom{n}{qn} \left( \sum_{Y\in \Y_1} \frac{\e(Y)}{\binom{n}{qn}} \right)^{\e(H)}\\
\mu \binom{n - h}{qn - h} &\geq \beta (qn)^{h - 2 \e(H)} \binom{n}{qn} \left( \frac{1}{8} q^2pn^2 \right)^{\e(H)}\\
\mu &\geq  \frac{\beta}{8^{\ell}} \frac{\binom{n}{qn}}{\binom{n - h}{qn - h}} (qn)^h p^{\e(H)} \\
\mu &\geq \beta' n^{h}p^{\e(H)}.
\end{align*} Thus as long as $|\Hc| <\beta' n^{h}p^{\e(H)}$,  we may add one $H$ to our collection and keep condition $(2)$. To see this, we notice that all pairs in $W$ would have the same degree as before, while each of the pairs not inside $S$ has its degree increased at most by $1$, we have that the degree of every configuration $(F, u)$ is at most $ \gamma  p^{\e(H) -\e(F)}n^{h - \ell}$ as long as
\begin{align*}
\frac{\gamma}{2} p^{\e(H) -\e(F)}n^{h - \ell} + 1 &< \gamma  p^{\e(H) -\e(F)}n^{h - \ell}\\
\frac{2}{ \gamma} n^{\ell - h }  &< p^{\e(H) - \e(F)}\\
p &> \left(\frac{2}{\gamma}\right)^{\frac{1}{ \e(H) - \e(F)}}n^{- \frac{h- \ell}{\e(H) - \e(F)}}.
\end{align*}

Similarly, we maintain $(3)$ as long as 

\begin{align*}
&1 \leq \frac{\beta' \ve}{2 (1 + \frac{2}{\ve})}p^{\e(H) - 1}n^{h - 2}\\
&p \geq \left( \frac{ 2 (1 + \frac{2}{\ve} )}{\beta' \ve }\right)^{\frac{1}{\e(H) -1}} n^{- \frac{h - 2}{\e(H) - 1}}\\
&p \geq \left(\frac{2}{\ve} + \frac{4}{\ve^2} \right)^{\frac{1}{\e(H) - 1}} \left( \frac{1}{\beta'}\right)^{\frac{1}{\e(H) - 1}}n^{- \frac{h - 2}{\e(H) - 1}}. 
\end{align*}
This will be dominated by the second bound on $p$ provided $\ve$ is constant with respect to $n$ by Theorem~\ref{problowerbound}. 

The result then holds with 

$$C' = \max\left\{\left( (16 C)^{\ell - 1}\left( 8 \beta' h \right)^{ 1 - \alpha}\right)^{ \frac{1}{(\ell - 1) + (1 - \e(F))(1 - \alpha)}} , 2^\frac{1}{\e(H) - \e(F)}\right\}.$$ 
\end{proof}

\section{Proofs of Main Results}

\begin{theorem}[Restatement of Theorem~\ref{mainmain}]\label{general}
Let $H$ be an $(\alpha, C, \beta)$-good graph with $h$ vertices, and let $F$ be a proper forest contained in $H$ with $\ell$ vertices. Let $H^*$ be formed from $H$ by taking $s$ copies of $H$ and gluing them along $F$. Let $\alpha' = \max\{ 1 - \frac{h - \ell}{\e(H) - \e(F)}, 1 - \frac{1 - \alpha}{\ell - 1 + (1 - \e(F))(1 - \alpha)}\}$. 
Then, there exists a constant $C'$ such that
$$ \ex(n, H^*) \leq C'n^{1 + \alpha'} $$ and moreover for $ H^*$ is $(\alpha', C', \beta'')$-supersaturated for $\beta'' = \frac{1}{4^{e(H)}}\left( \frac{\beta}{4 \cdot 8^{e(H)}(6K)^{\e(F)}}\right)^s $. 

\end{theorem}

\begin{proof} Let $K, C_{\ref{reduction}}$ be as in Theorem~\ref{reduction}. Let $C_{\ref{first-cleaning}}, \beta', n_0$ be returned by Lemma~\ref{first-cleaning} for $H, F$. Let $\gamma = \frac{\beta'}{4 sh (6K)^{\e(F)}}$. Let $g = g(n, H, F, 1, \gamma)$ be as in Lemma~\ref{first-cleaning} applied with $n, H, F,\gamma$, and $\ve = 1$. Let $p \geq g$. 

Let $C = \max\{n_0, C_{\ref{reduction}},C_{\ref{first-cleaning}}\}\cdot \max\{\gamma^{ -\frac{1}{\e(H) - \e(F)}}, \gamma^{\frac{\alpha - 1}{(\ell - 1) + (1 -\alpha)(1 - \e(F))}}\}$. Following Theorem~\ref{reduction}, it suffices to examine a graph $G$ which is $K$-almost-regular. Let $G$ be a $K$-almost-regular graph with $e(G) = 3pn^{2}$. Note that in such a graph $G$, $\Delta(G) \leq 6Kpn$, so the number of embeddings $\psi$ of $F$ into $G$ is at most $n^{\ell}(6Kp)^{\e(F)}$. Denote this family of $\psi: F \rightarrow G$ as $\F$.  Let $\mathcal{H}$ be a collection of $\varphi: H \rightarrow G$ returned by Lemma~\ref{first-cleaning} with $\gamma$ and $\ve = 1$. For each $\psi:F\rightarrow G$, let $\deg_{\mathcal{H}}(\psi) = |\{ \varphi \in \mathcal{H}: \varphi|_F = \psi\} |$. 

We divide the family $\F$ into two subfamilies, $$\F_{\text{heavy}} = \{ \psi : \deg_\mathcal{H}(\psi) \geq \frac{ \beta'}{2^{4} ( 6 K)^{\e(F)}} p^{\e(H) -\e(F)}n^{h - \ell} \}$$ and $$\F_{\text{light}}= \{ \psi : \deg_\mathcal{H}(\psi) < \frac{\beta'}{2^{4 } (6 K )^{\e(F)}} p^{\e(H) - \e(F)}n^{h - \ell} \}.$$ 
\begin{claim}
$$\sum_{\psi \in \F_{\rm{light}}} \deg_\mathcal{H}(\psi) < \frac{1}{2}  |\Hc|.$$
\end{claim}
\begin{proof}
This is not hard to show. Indeed, 

\begin{align*}\sum_{\psi \in \F_{\text{light}}} \deg_\mathcal{H}(\psi) &\leq |\mathcal{F}| \max \{\deg_\mathcal{H}(\psi):\psi\in \mathcal{F}_{\text{light}}\}  \\
&\leq n^{\ell}(6K p)^{\e(F)} \cdot \frac{\beta'}{2^{4}(6 K )^{\e(F)}} p^{\e(H) -\e(F)}n^{h - \ell}  \\
&<\frac{\beta'}{2^4} p^{\e(H)}n^h \leq \frac{1}{2}|\mathcal{H}|.
\end{align*}

\end{proof}

Thus, $\sum_{\psi \in \F_{\text{heavy}}} \deg_\mathcal{H}(\psi) \geq \frac{1}{2} |\mathcal{H}|$.

Fix $\psi \in \F_{\text{heavy}}$, then the number of labeled copies of $H^*$ rooted at $\psi(F)$ is at least 
$$\prod_{ 0 \leq i < s } \left( \deg_{\Hc}(\psi) -  i h  \gamma p^{\e(H) -\e(F)}n^{h - \ell} \right) .$$ Indeed, having picked $(i - 1)$ copies of $H$, they block no more than $(i - 1)h \gamma p^{\e(H) - \e(F)}n^{ h - \ell}$ copies of $H$ for the $i^{\text{th}}$ step by Lemma~\ref{first-cleaning} part (2). We may lower bound this by $\frac{1}{2^s} \deg_{\Hc}(\psi)^s$. Indeed, by choice of $\gamma$, we guaranteed that 
$$\deg_{\Hc}(\psi) \geq \frac{\beta'}{2^4 (6K)^{\e(F)}}p^{\e(H) - \e(F)} n^{h - \ell } =  2 sh \gamma p^{\e(H) - \e(F)} n^{h - \ell}.$$  Thus, the number of copies of $H^*$ in $G$ by Jensen's inequality is at least as follows, where in the fifth inequality we use $|\Hc| \geq \beta' p^{e(H)}n^h$ by Lemma~\ref{first-cleaning} (1). 
\begin{align*}
&\geq \sum_{\psi \in \F_{\text{heavy}}} \frac{1}{2^s} \deg_\Hc(\psi)^s \\
&\geq \frac{|\F_{\text{heavy}}|}{2^s} \left( \sum_{\psi \in \F_{\text{heavy}}} \frac{\deg_\Hc(\psi)}{|\F_{\text{heavy}}|}\right)^s\\
&\geq \frac{|\F_{\text{heavy}}|^{1 - s}}{2^s} \left( \sum_{\psi \in \F_{\text{heavy}}} \deg_\Hc(\psi)\right)^s\\
&\geq \frac{1}{2^{s}} ((6Kp)^{\e(F)}n^{\ell})^{s - 1} \left(\frac{1}{2} |\Hc|\right)^{s}\\
&\geq \frac{1}{2^{2s}} ((6Kp)^{\e(F)}n^{\ell})^{s - 1} \left( \beta' p^{e(H)} n^h\right)^{s}\\
&\geq \left( \frac{\beta'}{4(6K)^{\e(F)}}\right)^s p^{\e(F)} n^\ell( p^{\e(H) - \e(F)}n^{h - \ell} )^s.
\end{align*}
\end{proof}

We record the following corollary, though note that it is likely only tight when $T$ is a star. 
\begin{corollary}[Restatement of Corollary~\ref{maintree}]
Let $T$ be a tree on $t \geq 3$ vertices with $r$ leaves,  and let $T^q$ be formed by gluing $q$ copies of $T$ along the leaves.  Then, 
there exist constants $C, \beta$ such that $T^q$ is supersaturated with respect to $( 1 -\frac{1}{r},C, \beta)$. 
\end{corollary}
\begin{proof}
Let $T'$ be the tree formed by extending one leaf of $T$, and let $F$ be the forest containing this new edge $f$ and the remaining leaves. Let $T^*$ be the graph formed from $q$ copies of $T'$ identified along the forest $F$. 

Let $G$ be a $K$-almost-regular graph with $e(G) = pn^2 \geq C\max\{ n^{ 2 - \frac{t - r}{t - 1}}, n^{ 2 - \frac{1}{r}}\}$, with $C$ returned by Theorem~\ref{general} applied with $T'$ and $F$.  Note that the second bound is always larger as $t \geq r + 1$ and $2 - \frac{t - r - 1}{t - 1}$ is monotonically decreasing in $t$.

Then, Theorem~\ref{general} applied to $T'$ and $F$ in $G$ graph gives at least 
$$\Omega(n^{q(\vr(T) - r) + r + 1}p^{q \e(T) + 1})$$
copies of $T^*$ in $G$. Since any copy of $T^q$ gives no more than $2Kpn$ copies of $T^*$, we have that there are at least $\Omega(n^{q(\vr(T) - r) + r }p^{q \e(T)})$ copies of $T^q$ in $G$. By Lemma~\ref{reduction}, $T^q$ is $(1-\frac{1}{r}, C, \beta)$-supersaturated for some $C, \beta$. 
\end{proof}

The following immediately implies Theorem~\ref{maingluingalongedge}. 

\begin{theorem}\label{mixedfamilies}
Let $H_1, H_2, \dots ,H_t$ be graphs such that each $H_i$ is supersaturated with respect to $( \alpha_i, C_i, \beta_i)$. Fix an edge $f_i$ in each $H_i$, and let $\Hc$ be the set of all possible graphs formed by gluing $H_i$'s along $f_i$.

Letting $\alpha = \max_{1\leq i \leq t}(\alpha_i)$, then 

$$\ex(n, \Hc) = O(n^{1 + \alpha}).$$

In particular, if $|\Hc| = 1$, then the unique resultant graph $F \in \Hc$ is supersaturated with respect to $( \alpha,C, \beta)$ for some constants $C, \beta$.
\end{theorem}
\begin{proof}
Let $H_1, H_2, \dots, H_t$ be graphs such that each $H_i$ is $(\alpha_i, C_i, \beta_i)$-good. Let $s = \sum_{i = 1}^{t} v(H_i)$. Let $\ve = \frac{1}{2t^2}$ and $\gamma = \frac{\ve}{4st(1 + \frac{2}{\ve}) }\min_{i = 1}^t\left\{\beta_i' \right\}$. Let $G$ be a graph with at least $(1+\frac{2}{\ve})pn^2$ edges with $p\geq g(n, H_i, f_i, \ve, \gamma)$ for every $i \in [t]$.  For an edge $f'\in E(G)$, define $\deg_{\F_i}(f') = |\{ \varphi \in \F_i : \varphi(f_i) = f' \}|$. 

Let each $\F_i$ be a family of copies of $H_i$ in $G$ satisfying Lemma~\ref{first-cleaning} with respect to $\gamma, \ve, F = f_i, f = f_i$. Now, we clean $\F_i$ to obtain subfamilies $\F_i'$ with respect to the following rule: 

For every edge $f' \in E(G)$, if $\deg_{\F_i}(f') \leq \frac{\ve}{2} \frac{|\F_i|}{\e(G)}$ for some $i\in [t]$, then remove from every $\F_i$ the member $\varphi:H_i \rightarrow G$ such that $\varphi(f_i) = f'$. 

We claim that there are many copies of $H_i$'s left in $\F_i'$. In particular, only a few edges in $G$ are involved in the above process. 

For each $\F_i$, let $x_i$ be the number of edges in $G$ whose degree is less than $\frac{\ve}{2} \frac{|\F_i|}{\e(G)}$. Observe that 
$$|\F_i| \leq (1 + \ve)\frac{|\F_i|}{\e(G)}(e(G) - x_i) + \frac{\ve}{2} \frac{|\F_i|}{\e(G)} x_i. $$

Indeed, every $\psi: H_i \rightarrow G$ in $\F_i$ either sends $f_i$ to a high-degree edge or a low-degree edge. Since by Lemma~\ref{first-cleaning} part (3), for any $f'\in E(G)$ and every family $\F_i$, $\deg_{\F_i}(f') \leq (1 + \ve) \frac{|\F_i|}{\e(G)}$, the statement follows. Thus, $x_i \leq \ve \e(G)$. Since we selected $\ve = \frac{1}{2t^2}$, we have that at most $\frac{\e(G)}{2t}$ bad edges are involved in this process. Thus, $|\F_i'| \geq |\F_i| - (1 + \frac{1}{2t^2})\frac{ |\F_i|}{\e(G)}\frac{\e(G)}{2t} \geq (1 - \frac{1}{t})|\F_i|$. 

As we picked $\gamma \leq \frac{\ve}{4st(1 + \frac{2}{\ve})} \beta_i'$ for every $i$, we can greedily build copies of our desired graph. Select any member $\varphi \in \F_1'$. Let $\varphi(H_1)=H_1'$. Observe there are at least $(1 - \frac{1}{t})|\F_1|$ such $\varphi$ in $\F'_1$. Define $f := \varphi(f_1)$. Observe that $f$ has high minimum degree in every $\F_i'$ by construction, at least
\begin{align*}
&\geq \frac{\ve}{2} \frac{|\F_2|}{\e(G)} -  v(H_1) \gamma p^{\e(H_2) - 1}n^{v(H_2) - 2}\\
&\geq \frac{\ve}{2} \frac{|\F_2|}{\e(G)} -   \frac{\ve v(H_1)}{4st} \frac{|\F_2|}{\e(G)}\\
&\geq \frac{\ve}{4} \frac{|\F_2|}{\e(G)}
\end{align*} copies of $H_2$ which intersect $H_1'$ only at $f$. Pick such a copy $H_2'$.

The same argument works for $H_i$. After we have chosen $H_1', H_2', \dots ,H_{i - 1}'$, there are at least $\frac{\ve}{4} \frac{|\F_i|}{\e(G)}$ many copies of $H_i$ avoiding $H_1', H_2', \dots ,H_{i - 1}'$. 
Thus, the process finds at least $$\left( 1 - \frac{1}{t} \right) |\F_1| \left( \frac{\ve}{4} \right)^{t - 1}\prod_{i = 2}^{t} \frac{|\F_i|}{\e(G)}$$ copies of some graph in $\Hc$. Recalling that $|\F_i| \geq \beta_i' p^{\e(H_i)} n^{\vr(H_i)}$, we have found $\Omega(p^{-t + \sum_{i = 1}^t\e(H_i)}n^{-2t + 2 + \sum_{ i = 1}^t \vr(H_i)})$ copies of some graph in $\Hc$. Thus, $\ex(n, \Hc) = O(n^{ 1+ \alpha})$. If furthermore $|\Hc|  = 1$, then the unique $F \in \Hc$ is $(\alpha, C, \beta)$-supersaturated.

\end{proof}

We observe the following easy lemma on gluing an $\alpha$-supersaturated graph and a tree at a vertex. This is a counting version of the analogous Lemma $12$ in \cite{faudree1983class}. 

\begin{lemma}\label{gluewithtree}
Suppose $H$ is an $(\alpha, C, \beta)$-supersaturated graph. Let $T$ be any tree. Fix two vertices $v, v'$ such that $v \in V(H)$, $v' \in V(T)$. Then, the graph $H^*$ formed by identifying $v$ with $v'$ is $(\alpha, \max\{ 2C, 4(\vr(H) + \vr(T))\}, \frac{\beta}{2^{\e(H) +2\e(T)}})$-supersaturated. 
\end{lemma}

\begin{proof}

Let $C' = \max\{2 C, 4 (\vr(H) + \vr(T))\}$ and $G$ be a graph with at least $pn^2$ edges with $pn^2 \geq C' n^{ 1 + \alpha}$. Let $G'$ be the subgraph of $G$ formed by iteratively removing all edges containing a vertex $u$ such that $\deg(u) < \frac{1}{2} pn$. Note that $\e(G') \geq \frac{1}{2} \e(G) \geq  C n^{ 1 + \alpha}$.

Notice that given any vertex of the tree, namely $v'$, there is an one-degenerate ordering of the vertices starting at that vertex. Let $v_0, \dots ,v_{\e(T)}$ be one such ordering. We will iterate over this ordering, letting $H_i$ be the graph formed from $H$ and the tree induced by the first $i$ elements by identifying $v$ with $v_0$. 

In the base case $i = 0$, by our assumptions on $H$, there are at least $\frac{\beta}{2^{\e(H)}}p^{\e(H)}n^{\vr(H)}$ many copies of $H$. 

By induction, we may assume that we have found $\frac{\beta}{2^{\e(H) + 2i}}n^{\vr(H) + i}p^{\e(H) + i}$ many copies of $H_i$. Fix one such copy $H'$, with $u$ the image the parent of $v_{i + 1}$. To extend to a copy of $H_{i + 1}$, we must find an edge $e \in E(G')$ such that $e \cap V(H') = u$. By our choice of $G'$, $\deg(v) \geq \frac{1}{2} pn \geq 2 (\vr(H) + \vr(T)) n^{\alpha}$. Thus, we have that at least $\frac{1}{4} p n$ edges intersect $H'$ only in $u$. As every one of these edges gives at least one labeled copy of $H_{i + 1}$, we have at least $\frac{\beta}{2^{\e(H) + 2(i  + 1)}}n^{\vr(H)+ i + 1}p^{\e(H)+i + 1}$ labeled copies of $H_i$ in $G$. Observing that $H_{\e(T)} = H^*$ finishes the proof.

\end{proof}

This lemma together with Theorem~\ref{mixedfamilies} gives the best known result for the following sequence of graphs. Recall the definition of a cycle blow-ups of a tree. Let $T$ be a tree on $[n]$ and $\mathcal{C} = \{C_1, \dots C_n\}$ be a family of disjoint even cycles with for every edge $ij \in E(T)$, a pair of vertices $v_{ij}^i \in C_i$ and $v_{ij}^j$ marked. Let $T[\mathcal{C}]$ be the graph formed by from the graphs $C_i \in \mathcal{C}$ by identifying joining the vertices $v_{ij}^i$ and $v_{ij}^j \in C_j$ by vertex-disjoint paths of length $k_{ij}$ for every $ij \in E[T]$. 
\begin{theorem}[Restatement of Theorem~\ref{cycle-blowups}]
Let $\ell$ be any positive integer. Given $T$ a tree on $[n]$ and $\mathcal{C} = \{C_1, \dots C_n\}$ be a family of disjoint even cycles with $|C_i| \geq 2 \ell$ for all $i$, then $T[\mathcal{C}]$ is $(\frac{1}{\ell}, C, \beta)$-supersaturated for some $C, \beta$.
\end{theorem}
\begin{proof}
We proceed by induction on the number of vertices of $T$. The statement is true for all choices of $T, \mathcal{C}$ with $T$ having one vertex by Theorem 1.2 in \cite{jiang2020supersaturation,FS}. 

Assume the statement holds for all choices of $T, \mathcal{C}$ with $\vr(T) = r - 1$, we will show this implies it holds for all $T, \mathcal{C}$ with $\vr(T) = t$ vertices. Pick any leaf vertex $i$ of $T$, and let $T' := T \setminus{i}$. Let $j$ be the unique neighbor of $i$ in $T$.  By induction, the graph $ T'[\mathcal{C} \setminus \{C_i\}]$ is  $(\frac{1}{\ell}, C', \beta')$-supersaturated. Thus, by Lemma~\ref{gluewithtree}, the graph $H$ formed by gluing a path $P$ of length $k_{ij} + 1$ to $T'[\mathcal{C} \setminus \{C_i\}]$ at the vertex $v_{ij}^j$ is $(\frac{1}{\ell}, C'', \beta'')$ -supersaturated for some $C'', \beta''$. Then, by Theorem~\ref{mixedfamilies}, the unique graph $H'$ formed by gluing $H$  with a vertex disjoint copy of $C_i$ along any edge adjacent to $v_{ij}^i$ in $C_i$ and the last edge of $P$ is $(\alpha , C, \beta)$-good for some $C, \beta$ with $\alpha = \max\{\frac{1}{\ell}, \frac{2}{|C_i|}\} = \frac{1}{\ell}$. 
\end{proof}

\section{Acknowledgments}
We thank Istv\'an Tomon for pointing out Theorem 1.1 from \cite{kupavskii2022extremal}. We thank David Conlon for bringing our attention to references \cite{li2011logarithimic, szegedy2014information}. We also thank Tao Jiang for letting us know about~\cite{Qiu} and helpful comments.

\bibliography{bibliography}

\bibliographystyle{abbrv}

\end{document}